\documentclass[a4paper, 11pt]{amsart}
\usepackage[latin1]{inputenc}
\usepackage[ T1]{fontenc}
\usepackage[english]{babel}
\usepackage{amssymb}
\usepackage{amsmath}
\usepackage{amsthm}
\usepackage{amscd}
\usepackage{amsfonts}
\usepackage{stmaryrd}
\usepackage{pb-diagram}
\usepackage{epic,eepic,epsfig}
\usepackage{a4wide}
\usepackage{nextpage}
\usepackage{fancyhdr}

\usepackage[OT2,T1]{fontenc}
\DeclareSymbolFont{cyrletters}{OT2}{wncyr}{m}{n}
\DeclareMathSymbol{\Sha}{\mathalpha}{cyrletters}{"58}

\newtheorem{prop}{Proposition}[section]

\newtheorem{lem}[prop]{Lemma}
\newtheorem{thm}[prop]{Theorem}
\newtheorem{conj}[prop]{Conjecture}

\theoremstyle{definition}
\newtheorem{rem}[prop]{Remark}
\newtheorem{defin}[prop]{Definition}
\newtheorem{nota}[prop]{Notation}
\newtheorem{ques}[prop]{Question}

\newcommand{\Ima} {\mathop{\mathrm{Im}}}

\newcommand{\Card} {\mathop{\mathrm{Card}}}

\newcommand{\degr} {\mathop{\mathrm{deg}}}

\newcommand{\Spec} {\mathop{\mathrm{Spec}}}

\newcommand{\rank} {\mathop{\mathrm{rank}}}

\newcommand{\hFplus} {\mathop{h_{\mathrm{F}^+}}}
\newcommand{\Reg} {\mathop{\mathrm{Reg}}}

\begin{document}

\title{Heights and regulators of number fields and elliptic curves}
\author{{F}abien {P}azuki}
\address{Department of Mathematical Sciences, University of Copenhagen, Universitetsparken 5, 2100 Copenhagen, Denmark, and IMB, Universit\'e de Bordeaux, 351 cours de la Lib\'eration, 33405 Talence.}
\email{fabien.pazuki@math.u-bordeaux.fr}
\thanks{The author wishes to thank the Laboratoire Math\'ematiques de l'Universit\'e de Franche-Comt\'e, in Besan\c{c}on and specially Aur\'elien Galateau for the invitation to give a course in October 2013 on the Lang and Silverman conjecture. One goal of the present text is to explain in details a new possible consequence of this conjecture. Many thanks to Ga\"el R\'emond, Qing Liu, Pascal Autissier, Marc Hindry, Eduardo Friedman and the anonymous referee for their comments. The author is supported by ANR-10-BLAN-0115 Hamot and ANR-10-JCJC-0107 Arivaf.}
\maketitle

\vspace{0.5cm}

\noindent \textbf{Abstract.}
We compare general inequalities between invariants of number fields and invariants of elliptic curves over number fields. On the number field side, we remark that there is only a finite number of non-CM number fields with bounded regulator. On the elliptic curve side, assuming the height conjecture of Lang and Silverman, we obtain a Northcott property for the regulator on the set of elliptic curves with dense rational points over a number field. This amounts to say that the arithmetic of CM fields is similar, with respect to the invariants considered here, to the arithmetic of elliptic curves over a number field having a non Zariski dense Mordell-Weil group, \textit{i.e.} with rank zero.

{\flushleft
 }

\begin{center}
---------
\end{center}

\begin{center}
\textbf{Hauteurs et r\'egulateurs de corps de nombres et de courbes elliptiques}.
\end{center}

\noindent \textbf{R\'esum\'e.}
On compare des in\'egalit\'es entre invariants classiques de corps de nombres et de courbes elliptiques d\'efinies sur un corps de nombres. Dans le cas des corps de nombres, on remarque qu'il n'y a qu'un nombre fini de corps de nombres non-CM avec r\'egulateur born\'e. On obtient alors comme cons\'equence d'une conjecture de Lang et Silverman une propri\'et\'e de Northcott pour le r\'egulateur sur l'ensemble des courbes elliptiques sur un corps de nombres dont les points rationnels sont denses. Cela indique que l'arithm\'etique d'un corps CM est similaire, au sens des invariants consid\'er\'es ici, \`a celle des courbes elliptiques sur un corps de nombres dont le groupe de Mordell-Weil n'est pas dense au sens de Zariski, donc de rang nul.

{\flushleft
\textbf{Mots-Clefs:} Hauteurs, vari\'et\'es ab\'eliennes, r\'egulateurs, Mordell-Weil.\\
}

\begin{center}
---------
\end{center}

\thispagestyle{empty}

\section{Introduction}
Comparisons between number fields and abelian varieties defined over number fields are used to help in the understanding of the arithmetic of both. As a recent example of such a comparison one can read \cite{Hin}, where a conjectural asymptotic in the spirit of the Brauer-Siegel theorem for number fields is given for abelian varieties. In this work we will focus on inequalities in both settings, but restricting ourselves to abelian varieties of dimension $1$, \textit{i.e.} elliptic curves. Firstly, we gather some useful inequalities between classical invariants --- discriminant, degree, regulator --- of number fields. Secondly we want to understand how the inequality of Theorem \ref{reg disc} between the regulator and the discriminant of a number field finds its elliptic counterpart in Theorem \ref{reg height thm}, assuming the Lang-Silverman height conjecture. Let us stress here that Hindry and Silverman showed in \cite{HiSi3} that the Lang-Silverman conjecture for elliptic curves is implied by the ABC conjecture, see Remark \ref{ABC} in the sequel.

For some inspiring references about regulators of number fields, one can consult various texts, including Friedman \cite{Fri}, Friedman and Skoruppa \cite{FriSko}, Berg\'e and Martinet \cite{BeMa}, Silverman \cite{Sil4}, Zimmert \cite{Zim}, Remak \cite{Rema}, Odlyzko \cite{Odl1, Odl2}.

Let us now state some consequences of the studied inequalities. On the number field side, one has the following. First recall that a CM field is a totally imaginary degree 2 extension of a totally real extension of $\mathbb{Q}$.

\begin{thm}\label{reg fini intro}
There exists only a finite number of non-CM number fields with bounded regulator.
\end{thm} 

The proof is relatively short and essentially relies on existing inequalities. It will be given in section \ref{nb}. On the elliptic curve side, one obtains the following result as a corollary of Theorem \ref{reg height thm}.

\begin{thm}\label{reg height thm intro}
Assume the Lang-Silverman height conjecture \ref{LangSilV1}.
The set of $\overline{\mathbb{Q}}$-isomorphism classes of elliptic curves $E$, defined over a fixed number field $K$ with $E(K)$ Zariski dense in $E$ and bounded regulator is finite.
\end{thm}

The Zariski density of $E(K)$ is equivalent to $E$ having positive Mordell-Weil rank over $K$ because $E$ is simple.

Focusing on elliptic curves defined over number fields, one also obtains (without assuming any height conjecture) an inequality between the Faltings height of Definitions \ref{faltings} and \ref{chai cond}, the primes of bad reduction and the injectivity diameter, which is of independent interest for two reasons: there is no assumption of semi-stability and all constants are explicit.

\begin{thm}\label{height conductor injectivity} 
Let $K$ be a number field of degree $d$. Then for any elliptic curve $E$ defined over $K$, one has $$\hFplus(E/K) \geq \frac{1}{3\cdot12^8 d}\log N^{0}_{E/K}+\frac{1}{3d}\sum_{v\in{M_K^{\infty}}}d_v\rho(E_v,L_v)^{-2}-\frac{1}{3}\log\frac{\pi}{\pi-3}$$ where $N^0_{E/K}$ is the product of the norms of the primes of $K$ of bad reduction for $E$ and $\rho(E_v,L_v)$ is the injectivity diameter of $E_v(\mathbb{C})$ polarized by $L_v$, the set of archimedean places of $K$ is denoted by $M_K^{\infty}$ and $d_v=[K_v:\mathbb{Q}_v]$ is the local degree.
\end{thm}

In section 2 we give the definitions of the regulator of a number field, of the regulator of an elliptic curve and of the Faltings height, also called the differential height, of an elliptic curve. In section 3 we gather the inequalities concerning number fields. In section 4 we prove the inequalities concerning elliptic curves and we compare these inequalities with the number field case. We conclude in section 5 with a comparative table.

Another way of reading the inequalities in this text lies in the following remark: in the number field case, both the discriminant and the regulator are volumes that carry arithmetic information, but the regulator is a weak invariant on CM-fields. For elliptic curves over a number field, the regulator and the differential height are both arithmetic volumes, but the regulator is a weak invariant on elliptic curves with non Zariski dense Mordell-Weil group (\textit{i.e.} on rank zero elliptic curves). The inequalities of Theorem \ref{reg disc} and Theorem \ref{reg height thm} compare these volumes precisely enough to reveal the special cases. An interesting question would be to obtain similar results in the case of general abelian varieties. This will be the object of a future work.

\section{Definitions}

If $K$ is a number field, we denote by $d$ its degree over $\mathbb{Q}$ and by $M_{K}$ the set of all places of $K$. Let $M_{K}^{\infty}$ be the set of archimedean places. For any $v\in{M_K}$, we have $K_{v}$ the completion of $K$ with respect to the valuation $|.|_{v}$. One normalizes $|p|_{v}=p^{-1}$ for any finite place $v$ above a prime number $p$. Then the local degree will be $d_v=[K_v:\mathbb{Q}_v]$. We will write $\mathfrak{p}$ for a prime ideal. We will denote by $D_K$ the absolute discriminant of $K$.

Let $S$ be a set. We will say that a function $f: S\to \mathbb{R}$ verifies a Northcott property on $S$ if for any real number $B$, the set $\{P\in{S}\,\vert\, f(P)\leq B\}$ is finite.

\subsection{Regulators of number fields and elliptic curves}

\subsubsection{Number fields}
Let $K$ be a number field of degree $d=r_1+2r_2$, with $r_1$ real embeddings $\sigma_1, ..., \sigma_{r_1}$ and $r_2$ pairs of complex conjugate embeddings $(\sigma_{r_1+1}, \overline{\sigma_{r_1+1}}), ..., (\sigma_{r_1+r_2}, \overline{\sigma_{r_1+r_2}})$. We denote by $U_K$ the group of units in $K$. By Dirichlet's theorem, $U_K$ is a group of rank $r=r_K=r_1+r_2-1$. Let $\varepsilon_1, ..., \varepsilon_r$ be a fundamental base, \textit{i.e.} any unit $\varepsilon\in{U_K}$ can be uniquely written as $\varepsilon=z \varepsilon_1^{n_1}\cdots\varepsilon_r^{n_r}$ with $z$ a root of unity in $K$ and $n_1, ..., n_r\in{\mathbb{Z}}$. The norm of any $\alpha\in{K}$ can be written as
$$N_{K/\mathbb{Q}}(\alpha)=\sigma_1(\alpha)\ldots\sigma_d(\alpha)=\sigma_1(\alpha)\ldots\sigma_{r_1}(\alpha)|\sigma_{r_1+1}(\alpha)|^2\ldots|\sigma_{r_1+r_2}(\alpha)|^2.$$
Let $\lambda$ be the logarithmic map given by
$$\begin{matrix}\lambda:&K&\rightarrow&\mathbb{R}^{r_1+r_2}\\&\alpha&\mapsto&\big(\log|\sigma_1(\alpha)|,\cdots,\log|\sigma_{r_1}(\alpha)|,2\log|\sigma_{r_1+1}(\alpha)|,\ldots,2\log|\sigma_{r_1+r_2}(\alpha)|\big),\end{matrix}$$
Moreover let $d_1=...=d_{r_1}=1$ for the real embeddings and $d_{r_1+1}=...=d_{r_1+r_2}=2$ for the complex embeddings -- so $d_i=d_{v_i}=[\overline{K_{v_i}}:\mathbb{Q}]$ for each $i\in\{1, ..., r_1+r_2\}$ and the first $r_1$ embeddings are the real ones.
Let $H=\{(x_1, ..., x_{r_1+r_2})\in{\mathbb{R}^{r_1+r_2}\,\vert\, x_1+...+x_{r_1+r_2}=0}\}$. We remark that $\lambda(U_K)\subset H$ because the norm of any unit is $\pm 1$, hence the logarithm of the absolute norm is $0$.
We are interested in the volume of the lattice spanned by the image of the units of $K$ in $H$. The vector $e=(1,...,1)$ is orthogonal to the hyperplane $H$ with respect to the canonical euclidean structure of $\mathbb{R}^{r_1+r_2}$, hence one has the following volume equality $$\mathrm{Vol}(\mathbb{R}^{r_1+r_2}/\lambda(U_K)\oplus \mathbb{Z}e)=\mathrm{Vol}(H/\lambda(U_K))\sqrt{r_1+r_2}.$$The volume of a fundamental domain of the lattice $\lambda(U_K)$ in the hyperplane $H$ is given by $\mathrm{Vol}(H/\lambda(U_K))$, where the volume is induced by the euclidean structure on $\mathbb{R}^{r_1+r_2}$. This motivates the following definition.

\begin{defin}\label{reg field}
The regulator $R_K$ of $K$ is defined in the following way $$R_K=\Big\vert \mathrm{det}\left[\begin{array}{cccc}
d_1\log\vert\sigma_1(\varepsilon_1)\vert & d_2\log\vert\sigma_2(\varepsilon_1)\vert & ... & d_{r_1+r_2}\log\vert\sigma_{r_1+r_2}(\varepsilon_1)\vert\\
d_1\log\vert\sigma_1(\varepsilon_2)\vert & d_2\log\vert\sigma_2(\varepsilon_2)\vert & ... & d_{r_1+r_2}\log\vert\sigma_{r_1+r_2}(\varepsilon_2)\vert\\
... & ... & ...  & ...\\
d_1\log\vert\sigma_1(\varepsilon_r)\vert & d_2\log\vert\sigma_2(\varepsilon_r)\vert & ... & d_{r_1+r_2}\log\vert\sigma_{r_1+r_2}(\varepsilon_r)\vert\\
(r_1+r_2)^{-1} & (r_1+r_2)^{-1} & ... & (r_1+r_2)^{-1}\\
\end{array}\right] \Big\vert,$$
where the square matrix is of size $(r_1+r_2)$.
\end{defin}

\vspace{0.5cm}

\noindent Remark that one has $R_K=\mathrm{Vol}(\mathbb{R}^{r_1+r_2}/\lambda(U_K)\oplus \mathbb{Z}e)(r_1+r_2)^{-1}=\mathrm{Vol}(H/\lambda(U_K))(r_1+r_2)^{-1/2}$, so the regulator is a volume obtained in the whole space, divided by the square root of the dimension of the whole space. An equivalent definition used in the litterature is

$$R_K=\Big\vert \mathrm{det}\left[\begin{array}{cccc}
d_1\log\vert\sigma_1(\varepsilon_1)\vert & d_2\log\vert\sigma_2(\varepsilon_1)\vert & ... & d_{r}\log\vert\sigma_{r}(\varepsilon_1)\vert\\
d_1\log\vert\sigma_1(\varepsilon_2)\vert & d_2\log\vert\sigma_2(\varepsilon_2)\vert & ... & d_{r}\log\vert\sigma_{r}(\varepsilon_2)\vert\\
... & ... & ...  & ...\\
d_1\log\vert\sigma_1(\varepsilon_r)\vert & d_2\log\vert\sigma_2(\varepsilon_r)\vert & ... & d_{r}\log\vert\sigma_{r}(\varepsilon_r)\vert\\
\end{array}\right] \Big\vert,$$
where one deletes the last column and the last line of the former matrix. To show equality of the two determinants, take the $(r_1+r_2)^2$ determinant from above and sum over all the columns to get a last column full of zeros (because the logarithm of the absolute norm of a unit is $0$) except for the last term at the bottom line where one gets $(r_1+r_2)(r_1+r_2)^{-1}=1$, and develop the determinant. By the same argument one knows that the final value doesn't depend on the column one erases, the only difference would be the sign which is removed by the absolute value outside the determinant.

This definition is the same as the one found in \cite{Neu} page 43, if one remarks that his $\lambda(.)$ is defined through a factorization of his $l(.)$ page 34 that has the same normalization for the metrics given here. 

We end this section by recalling that a CM field $K$ is a totally imaginary quadratic extension of a totally real number field.

\subsubsection{Elliptic curves}
Let $E/K$ be an elliptic curve over a number field $K$ polarized by $L=3(O)$ where $O$ is the neutral element. Let $m_K$ be the Mordell-Weil rank of $A(K)$. Let $\hat{h}=\hat{h}_{E,L}$ be the N\'eron-Tate height associated with the pair $(E,L)$. Let $<.,.>$ be the associated bilinear form, given by $$<P,Q>=\frac{1}{2}\Big(\hat{h}(P+Q)-\hat{h}(P)-\hat{h}(Q)\Big)$$ for any $P,Q\in{E(\bar{\mathbb{Q}})}$. 

\begin{defin}\label{reg abvar}
Let $P_1, ..., P_{m_K}$ be a basis of the Mordell-Weil group $E(K)$ modulo torsion. The regulator of $E/K$ is defined by $$\Reg(E/K)=\vert \det(<P_i,P_j>_{1\leq i,j\leq m_K})\vert,$$ where by convention an empty determinant is equal to $1$.
\end{defin}

\subsection{The Faltings height}

Let $E$ be an elliptic curve defined over a number field $K$. Let $S=\Spec({\mathcal O}_K)$, where
${\mathcal O}_K$ is the ring of integers of $K$ and let $\pi\colon {\mathcal E}\longrightarrow S $
be the N\'eron model of $E$ over $S$. Denote by $\varepsilon\colon S\longrightarrow {\mathcal E}$ the zero section of $\pi$ and by
$\omega_{{\mathcal E}/S}$ the sheaf of relative differentials
$$\omega_{{\mathcal E}/S}:=\varepsilon^{\star}\Omega_{{\mathcal
E}/S}\simeq\pi_{\star}\Omega_{{\mathcal E}/S}\;.$$

For any archimedean place $v$ of $K$, denote by $\sigma$ an embedding of $K$ in $\mathbb{C}$ associated to $v$, then the corresponding line bundle
$$\omega_{{\mathcal E}/S,\sigma}=\omega_{{\mathcal E}/S}\otimes_{{\mathcal O}_K,\sigma}\mathbb{C}\simeq H^0({\mathcal
E}_{\sigma}(\mathbb{C}),\Omega_{{\mathcal E}_\sigma}(\mathbb{C}))\;$$
can be equipped with a natural $L^2$-metric $\Vert.\Vert_{v}$ defined by
$$\Vert\alpha\Vert_{v}^2=\frac{i}{2}\int_{{\mathcal
E}_{\sigma}(\mathbb{C})}\alpha\wedge\overline{\alpha}\;.$$

The ${\mathcal O}_K$-module of rank one $\omega_{{\mathcal E}/S}$, together with the hermitian norms
$\Vert.\Vert_{v}$ at infinity defines an hermitian line bundle 
$\overline{\omega}_{{\mathcal E}/S}=({\omega}_{{\mathcal E}/S}, (\Vert .\Vert_v)_{v\in{M_K^\infty}})$ over $S$, which has a well defined Arakelov degree
$\widehat{\degr}(\overline{\omega}_{{\mathcal E}/S})$. Recall that for any hermitian line
bundle $\overline{\mathcal L}$ over $S$, the Arakelov degree of $\overline{\mathcal L}$
is defined as
$$\widehat{\degr}(\overline{\mathcal L})=\log\Card\left({\mathcal L}/{s{\mathcal
O}}_K\right)-\sum_{v\in{M_{K}^{\infty}}}d_v\log\Vert
s\Vert_{v}\;,$$
where $s$ is any non zero section of $\mathcal L$ (the result does not depend on the choice of $s$ in view of the product formula). 

We now give the definition of the classical Faltings height.

\begin{defin}\label{faltings}  The Faltings height of $E/K$ is defined as
$$h_F(E/K):=\frac{1}{[K:\mathbb{Q}]}\widehat{\degr}(\overline{\omega}_{{\mathcal
E}/S}).$$
\end{defin}

We like to think of the Faltings height as a volume, in particular as a positive real number. An inequality of Bost reads $h_F(E/K)\geq -\frac{1}{2}\log 2\pi^2$ (a detailed proof being accessible in \cite{GauR}), we thus would like to rescale the height slightly in the following way:

\begin{nota}
Let $\hFplus(E/K)=h_F(E/K)+\frac{1}{2}\log(2\pi^2).$
\end{nota}

We now obtain a non-negative number that can be interpreted as a volume. This little change of normalization also helps simplifying some formulas in the sequel. We will continue to call this number the Faltings height or the differential height.

The Faltings height doesn't depend on any choice of polarization on $E$. When $E/K$ is semi-stable, this height only depends on the $\overline{\mathbb{Q}}$-isomorphism class of $E$. If $E/K$ is not semi-stable, one may use Chai's base change conductor, take formula (\ref{base change conductor}) in the sequel as a complementary definition. The definition can be extended to abelian varieties of any fixed dimension. To see that it is really a height (meaning verifying the Northcott property), see for instance \cite{Falt} Satz~1, page 356 and 357. This can also be derived from the comparison of heights in \cite{Paz3} (an article based on former work by Bost and David) and the work of Mumford on theta structures. 

If $K'/K$ is a number field extension, one has $\hFplus(E/K')\leq \hFplus(E/K)$. If $E/K$ is semi-stable, one defines the stable height by $\hFplus(E/\overline{\mathbb{Q}}):=\hFplus(E/K)$, which is invariant by number field extension. Let us state the following formula (proved in \cite{CorSil} page 254) which will be crucial in the sequel:

\begin{thm}(Silverman)\label{elliptique}
Let $E/K$ be a semi-stable elliptic curve over a number field $K$ of degree $d$. One has
\[
\hFplus(E/K)=\frac{1}{12d}\left[\log N_{K/\mathbb{Q}}(\Delta_E)-\sum_{v\in{M_K^{\infty}}}d_v \log\Big(\vert \Delta(\tau_v) \vert (2\Ima\tau_v)^6\Big)\right],
\]
where $\Delta_E$ is the minimal discriminant of $E$; at archimedean places one chooses $\tau_v$ in the upper half plane such that $E(\overline{K}_v)\simeq \mathbb{C}/\mathbb{Z}+\tau_v\mathbb{Z}$ and $\Delta(\tau_v)=q\prod_{n=1}^{+\infty}(1-q^n)^{24}$ is the modular discriminant, with $q=\exp(2\pi i\tau_v)$.
\end{thm}

We recall that $\vert\Delta(\tau_v)\vert(\Ima\tau_v)^6$ is independent of the choice of $\tau_v$ in his orbit under the classical action of $\mathrm{SL}_2(\mathbb{Z})$ on the upper half plane. At finite places we focus on the bad reduction locus with the following classical quantities.

\begin{defin}
Let $E$ be an elliptic curve over a number field $K$. Let us denote by $N^0_{E/K}$ the norm of the product of the bad reduction primes of $E$. 
\end{defin}

\begin{defin}
Let $E$ be an elliptic curve over a number field $K$. We say that a prime $\mathfrak{p}$ of $\mathcal{O}_K$ is of \emph{stable bad reduction} if for any number field extension $K'/K$ and any prime $\mathfrak{p}'$ of $\mathcal{O}_K$ above $\mathfrak{p}$, the elliptic curve $E/K'$ has bad reduction at $\mathfrak{p}'$. If the prime $\mathfrak{p}$ is a prime of bad reduction for $E/K$ that is not a stable one, it will be referred to as an \emph{unstable bad prime} of $E/K$.
\end{defin}

Regarding archimedean places, let us recall what the injectivity diameter is. 
\begin{defin}
Let $E$ be a complex elliptic curve. Let $L$ be a polarization on $E$. Let $T_E$ be the tangent space of $E$ and $\Gamma_E$ its period lattice and $H_L$ the associated Riemann form on $T_E$. The injectivity diameter is the positive number $\displaystyle{\rho(E,L)=\min_{\gamma\in{\Gamma\backslash \{0\}}}\sqrt{H_L(\gamma,\gamma)}}$, \textit{i.e.} the first minimum in the successive minima of the period lattice of $E$.
\end{defin}

If one writes $(E,L)\simeq (\mathbb{C}/\mathbb{Z}+\tau\mathbb{Z}, H_{\tau})$ as polarized complex abelian varieties of dimension 1, where $\tau$ is a complex period verifying $\mathrm{Im}\tau>\sqrt{3}/2$ and $\vert\mathrm{Re}\tau\vert \leq1/2$, then one has $$\rho(E,L)=(\mathrm{Im}\tau)^{-1/2}.$$

We will now carry on a comparison between inequalities relating the number field invariants and inequalities relating elliptic curve invariants.

\section{Number fields}\label{nb}

We start this paragraph by gathering some unconditional inequalities between different classical invariants of number fields, the degree, the absolute discriminant and the regulator. One may consult \cite{Sam} for a reference. We will then give the proof of Theorem \ref{reg fini intro}.

\begin{thm} (Hermite-Minkowski)\label{HM}
Let $K$ be a number field of discriminant $D_K$ and degree $d$. Then
\[
\vert D_K\vert^{1/2}\geq \left(\frac{\pi}{4}\right)^{r_2}\frac{d^d}{d!}
\]
and if $d\geq 2$
\[
\vert D_K\vert\geq \frac{\pi}{3}\left(\frac{3\pi}{4}\right)^{d-1}.
\]
\end{thm}

\begin{rem}
For a better bound when $d$ is large, see Odlyzko \cite{Odl1} who gives $$\vert D_K\vert^{1/d} \geq 60^{r_1/d} 22^{2r_2/d} + o(1)$$ when $d\to+\infty$.
\end{rem}

The following classical theorem will be used in the sequel.

\begin{thm} (Hermite)\label{Her}
There is only a finite number of number fields with bounded discriminant.
\end{thm}

The following two theorems compare the regulator with the degree of the number field. The second statement is even more precise because it gives a lower bound on a \textit{relative} regulator.

\begin{thm}(Friedman, \cite{Fri} page 620, Corollary)
Let $K$ be a number field of regulator $R_K$, degree $d$ and number of roots of unity $w_K$. Let $r_1$ be the number of real embeddings of $K$. Then
$$\frac{R_K}{w_K}\geq 0.0031\exp(0.241d+0.497r_1).$$
\end{thm}

\begin{thm}(Friedman, Skoruppa, \cite{FriSko})\label{reg reg}
There exists absolute constants $c_1>0$ and $c_2>1$ such that for any number field extension $L/K$ the ratio of their regulators satisfies $$\frac{R_L}{R_K}\geq \big(c_1 c_2^{[L:K]}\big)^{[K:\mathbb{Q}]},$$ and one can take $c_1=1/(11.5)^{39}$ and $c_2=1.15$.
\end{thm}

The next step is important, it provides a lower bound for the regulator in terms of the discriminant and of the degree of the field. One needs to pay attention, the inequality is trivial in some cases that are analyzed in the sequel.

\begin{thm}(Silverman \cite{Sil4}, Friedman \cite{Fri})\label{reg disc}
Let $K$ be a number field of discriminant $D_K$, regulator $R_K$ and degree $d$. Let $r_K$ denote the unit rank of $K$ and $r_0$ the maximum of the unit ranks of proper subfields of $K$. Then there exists a universal constant $c_3>0$ such that
\[
R_K\geq c_3 d^{-2d}\left(\log\frac{\vert D_K\vert}{d^d}\right)^{r_K-r_0}.
\]
\end{thm}

The next proposition explains the case $r_K=r_0$.

\begin{prop} \label{CM}
One has $r_K=r_0$ if and only if $K$ is a CM field.
Moreover, if $K$ is a CM field and $K_0$ its maximal totally real subfield, then $$\frac{R_K}{R_{K_0}}=2^{s},\quad \mathrm{with}\, r_0-1\leq s \leq r_0.$$
\end{prop}

\begin{proof}
Let $K$ be a number field of rank $r_K$ and $K_0$ be a strict subfield realizing $r_0=\rank(K_0)$. Then one has $r_K=r_1(K)+r_2(K)-1$ and $r_0=r_1(K_0)+r_2(K_0)-1$. If $r_K=r_0$, it gives
\begin{equation}\label{rr}
r_1(K)+r_2(K)=r_1(K_0)+r_2(K_0)
\end{equation}
on the one hand, and by the degree formula $r_1(K)+2r_2(K)=[K:K_0](r_1(K_0)+2r_2(K_0))$ on the other hand. It implies by substitution that
$$r_1(K)+2r_2(K)=[K:K_0](r_1(K)+r_2(K)+r_2(K_0))=[K:K_0](r_1(K)+2r_2(K)+r_2(K_0)-r_2(K)),$$ hence,
 \begin{equation}\label{KK} [K:K_0](r_2(K)-r_2(K_0))=([K:K_0]-1)(r_1(K)+2r_2(K))>0 \end{equation} as $[K:K_0]\geq 2$ for the last inequality and so $r_2(K)>0$. Next, using $r_2(K_0)\geq 0$ and $r_1(K)\geq0$ one obtains in (\ref{KK}) $$[K:K_0]r_2(K)\geq (2[K:K_0]-2)r_2(K)$$ which implies $2\geq [K:K_0]$, so $[K:K_0]=2$. Back to (\ref{KK}), one obtains $$2r_2(K)-2r_2(K_0)=r_1(K)+2r_2(K),$$ so $-2r_2(K_0)=r_1(K)$ and both numbers $r_1$ and $r_2$ are non-negative, hence must both be zero. So from (\ref{rr}) one obtains $r_2(K)=r_1(K_0)$. Finally we have $K_0$ with only real embeddings and $K$ a purely imaginary degree two extension of $K_0$, \textit{i.e.} $K$ is a CM field. 
 
 Reciprocally, if $K$ is a CM field, let $K_0$ be its maximal totally real subfield. Then we will show below that any base of $U_{K_0}$ modulo $\{\pm 1\}$ lifts to a base of a subgroup of index 1 or 2 in $U_K$ modulo roots of unity, hence the ranks are equal, which implies $r_K=r_0$.

Take a unit $\varepsilon$ in $K$. One will denote by a bar the complex conjugation from $K/K_0$. Consider $\bar{\varepsilon}/\varepsilon$. It is also a unit and its image by any of the complex embeddings of $K$ is of module 1. As a unit, it is also an algebraic integer. Hence just by using the definition of the height of an algebraic number, one has $h(\bar{\varepsilon}/\varepsilon)=0$. By a theorem of Kronecker, one obtains that $\bar{\varepsilon}/\varepsilon$ is a root of unity in $K$. Let us denote by $\mu_K$ the set of roots of unity of $K$. Writing $\varepsilon \bar{\varepsilon}=\varepsilon^2\frac{\bar{\varepsilon}}{\varepsilon}\in{K_0}$, one obtains that up to a torsion element, $\varepsilon^2\in{K_0}$. Moreover if $\bar{\varepsilon}/\epsilon=\zeta^2$ where $\zeta\in{\mu_K}$, then $\bar{\epsilon}\zeta^{-1}=\epsilon \zeta$, so $\overline{\varepsilon \zeta}=\pm \varepsilon \zeta$. Hence the mapping $\varepsilon\mapsto\bar{\varepsilon}/\varepsilon$ is an injection $$U_K/\mu_K U_{K_0}\hookrightarrow \mu_K/\mu_K^2,$$
so the index of $U_K$ modulo torsion over $U_{K_0}$ modulo torsion is bounded by 2. Now if $\varepsilon_1, ..., \varepsilon_r$ is a basis of $U_{K_0}$ modulo $K_0$-roots of unity, it will lift to a basis of a subgroup of index $1$ or $2$ of $U_K$ modulo $K$-roots of unity. In the determinant defining the regulator given in Definition \ref{reg field} all embeddings are real for $K_0$, hence $d_{i}=1$ for all $i\in\{1,...,r_0\}$, whereas all embeddings are complex for $K$, hence $d_{i}=2$ for all $i\in\{1,...,r_K\}$. Now as $r_0=r_K$ we get the result by applying Lemma 4.15 page 41 of \cite{Wash}.
\end{proof}

One finds the comparison of regulators in \cite{Wash} page 41, with the $(r_1+r_2-1)^2$-matrix to define the regulator (\textit{loc.cit.} same page). The second part of the proof here essentially follows his strategy.

\begin{rem}\label{disc CM}
Proposition \ref{CM} shows that the regulator of a CM field is bounded in terms of its maximal totally real subfield, hence one cannot expect to obtain $R_K\gg \log\vert D_K\vert$ in that case, because $\vert D_K\vert$ may not be bounded. Take for instance a family of quadratic imaginary extensions of $\mathbb{Q}$, then $\vert D_K\vert$ tends to infinity but $R_K$ is bounded.
\end{rem}

We thus obtain the following proof of Theorem \ref{reg fini intro}. In a nutshell: the regulator verifies a Northcott property on the collection of all non-CM fields.

\begin{proof}
Let us take a family of non-CM number fields $K$ with bounded regulator. Then using Theorem \ref{Her} one has that the degree $d$ is bounded. Moreover using Theorem \ref{reg disc} and the fact that $r_K-r_0\neq 0$ by Proposition \ref{CM}, then the discriminant is also bounded because we have just seen that  the degree $d$ is bounded. This gives finiteness by Hermite's Theorem \ref{HM}. 
\end{proof}

If one needs an asymptotic result concerning the discriminant and the regulator of number fields, let us also quote the classical Brauer-Siegel theorem.

\begin{thm} (Brauer-Siegel)
Consider a family of number fields $K$ of bounded degree and let $h_K$ denote the class number of $K$. Then as $\vert D_K\vert$ tends to infinity, one has for any $\varepsilon >0$ the inequalities $$\vert D_K\vert^{1/2-\varepsilon}\ll h_K R_K \ll \vert D_K\vert^{1/2+\varepsilon}.$$
\end{thm}

If the reader is interested in a comparison between this theorem and the behavior of some invariants of abelian varieties, a detailed study is accessible in \cite{Hin}.

\section{Elliptic curves}

We give in this section a lower bound of the differential height by the norm of the product of the bad reduction primes in the semi-stable case, then we obtain the result in the general case, hence deriving a proof of Theorem \ref{height conductor injectivity}. This implies an upper bound on the Mordell-Weil rank of elliptic curves over number fields in terms of the differential height. Using this and the Lang-Silverman conjecture, we obtain a proof of Theorem \ref{reg height thm}, which implies Theorem \ref{reg height thm intro}.

\subsection{Bad reduction primes}

Our aim in this paragraph is to compare the height of $E$ and the norm of the product of the bad reduction primes of $E$ over the base field $K$. We will first prove the inequality in the semi-stable case and then derive it in general using some base change properties in the next paragraph. The following proposition gives the result in the semi-stable case. 

\begin{prop}\label{semi stable height conductor}
Let $E/K$ be a semi-stable elliptic curve defined over a number field $K$ of degree $d$. Then  $$\hFplus(E/K)\geq \frac{1}{12d}\log N^0_{E/K}.$$
\end{prop}

\begin{proof}
 One has the exact formula from Theorem \ref{elliptique} 
\[
\hFplus(E/K)=\frac{1}{12d}\left[\log N_{K/\mathbb{Q}}(\Delta_E)-\sum_{v\in{M_K^{\infty}}}d_v \log\Big(\vert \Delta(\tau_v) \vert (2\Ima\tau_v)^6\Big)\right],
\]
where $\Delta_E$ is the minimal discriminant of the curve, one may choose $\tau_v$ a period in the fundamental domain such that $E(\bar{K}_v)\simeq\mathbb{C}/\mathbb{Z}+\tau_v\mathbb{Z}$ and $\Delta(\tau_v)=q\prod_{n=1}^{+\infty}(1-q^n)^{24}$ is the modular discriminant, with $q=\exp(2\pi i\tau_v)$. A direct analytic estimate using $\Ima\tau_v\geq \sqrt{3}/2$ provides us with

\[
\sum_{n=1}^{+\infty}\log\vert1-e^{2i\pi\tau_v n}\vert\leq \sum_{n=1}^{+\infty}\log (1+e^{-\sqrt{3}\pi n})\leq 0.005
\]
hence
\[
-2\pi\Ima\tau_v +24\sum_{n=1}^{+\infty}\log\vert1-e^{2i\pi\tau_v n}\vert+6\log(2\Ima\tau_v)\leq 0,
\]
hence $-\log(\vert \Delta(\tau_v)\vert (2\Ima\tau_v)^6)\geq 0$ and

\begin{equation}\label{elliptic}
\hFplus(E/K)\geq \frac{1}{12d}\log N_{K/\mathbb{Q}}(\Delta_E) \geq \frac{1}{12d}\log N^0_{E/K},
\end{equation}
because the bad reduction primes divide the discriminant.
\end{proof}

\subsection{Reducing to the semi-stable case}

We explain in this section how to use base change properties to jump from the semi-stable case to the general case. We will use the following definition.

\begin{defin}\label{chai cond}
Let $E$ be an elliptic curve defined over a discrete valuation field $K_\mathfrak{p}$ and let $K_{\mathfrak{p}'}'$ be a finite extension of $K_\mathfrak{p}$ where $E$ has semi-stable reduction, with ramification index $e_{K_{\mathfrak{p}'}'/K_\mathfrak{p}}$, where $\mathfrak{p}'$ is a prime above $\mathfrak{p}$, and $\omega_{E/K_{\mathfrak{p}}}$ the sheaf of differentials.
Let $$c(E)=\frac{1}{e_{K_{\mathfrak{p}'}'/K_{\mathfrak{p}}}}\mathrm{length}_{\mathcal{O}_{K_{\mathfrak{p}'}'}}\frac{\Gamma(\Spec(\mathcal{O}_{K_{\mathfrak{p}}}),\omega_{E/K_{\mathfrak{p}}})\otimes \mathcal{O}_{K_{\mathfrak{p}'}'}}{\Gamma(\Spec(\mathcal{O}_{K_{\mathfrak{p}'}'}), \omega_{E/K_{\mathfrak{p}'}'})},$$
 where $\Gamma(.,.)$ stands for global sections. For any prime $\mathfrak{p}$ in $\mathcal{O}_K$, we define the local base change conductor at $\mathfrak{p}$ by setting $c(E,\mathfrak{p})=c(E_{K_\mathfrak{p}})$.
\end{defin}

This conductor was defined by Chai in \cite{Cha}. It verifies the two following key properties.
\begin{prop}
Let $E$ be an elliptic curve defined over a discrete valuation field $K_\mathfrak{p}$ and let $K_{\mathfrak{p}'}'$ be a finite extension of $K_\mathfrak{p}$ where $E$ has semi-stable reduction with base change conductor $c(E,{\mathfrak{p}})$. Then
\begin{enumerate}
\item one has $c(E,\mathfrak{p})=0$ if and only if $E/K_{\mathfrak{p}}$ has semi-stable reduction,
\item if $c(E,\mathfrak{p})\neq 0$, then $c(E,\mathfrak{p})\geq 1/[K_{\mathfrak{p}'}':K_{\mathfrak{p}}]$.
\end{enumerate}
\end{prop}

\begin{proof}
To show the first point, let us start by assuming that $E/K_{\mathfrak{p}}$ has semi-stable reduction. Then if one denotes by $\mathcal{E}^0_{\mathcal{O}_{K_\mathfrak{p}}}$ the identity component of the N\'eron model of $E$ over $K_{\mathfrak{p}}$ one has $\mathcal{E}^0_{\mathcal{O}_{K_\mathfrak{p}}}\otimes\mathcal{O}_{K_{\mathfrak{p}'}'}\simeq \mathcal{E}^0_{\mathcal{O}_{K_{\mathfrak{p}'}'}}$ by Corollaire 3.3 page 348 of SGA 7.1 \cite{SGA}, hence the differentials are the same and $c(E, {\mathfrak{p}})=0$.

Reciprocally, one still has a map $\Phi: \mathcal{E}^0_{\mathcal{O}_{K_\mathfrak{p}}}\otimes\mathcal{O}_{K_{\mathfrak{p}'}'} \to\mathcal{E}^0_{\mathcal{O}_{K_{\mathfrak{p}'}'}}$ by the N\'eron property. As $c(E, {\mathfrak{p}})=0$, the Lie algebras are the same and as $\Phi$ is an isomorphism on the generic fibers, $\Phi$ is birational. On the special fiber, $\Phi$ has finite kernel and is thus surjective because the dimensions are equal, here again because $c(E, {\mathfrak{p}})=0$.

Overall, $\Phi$ is quasi-finite and birational. As $\mathcal{E}_{\mathcal{O}_{K_\mathfrak{p}'}'}$ is normal, by Zariski's Main Theorem found in Corollary 4.6 page 152 of \cite{Liu}, $\Phi$ is an open immersion. Hence $\Phi$ is surjective and an open immersion, hence an isomorphism. This implies that $E/K_{\mathfrak{p}}$ is semi-abelian. 

For the second point, as the length is non zero it must be at least one and we may upper bound $e_{K_{\mathfrak{p}'}'/K_{\mathfrak{p}}}\leq [K_{\mathfrak{p}'}' :K_{\mathfrak{p}}]$.
\end{proof}

Let $Uns$ denote the set of unstable primes of $E$ over $K$. Let $K'$ be a number field extension of $K$ over which $E$ has semi-stable reduction everywhere. One then has
\begin{equation}\label{base change conductor}
\hFplus(E/K')-\hFplus(E/K)=-\frac{1}{[K':K]}\sum_{\mathfrak{p}\in{Uns}} c(E,\mathfrak{p})\log \mathcal{N}(\mathfrak{p}).
\end{equation}

\begin{prop}\label{height conductor}
Let $K$ a number field of degree $d$. For any elliptic curve (not necessarily semi-stable) $E$ defined over $K$ one has $$\hFplus(E/K) \geq \frac{1}{12^8 d}\log N^{0}_{E/K}.$$
\end{prop}

\begin{proof}
 Let $N^{st}_{E/K}$ be the product of the norms of primes with semi-stable bad reduction. Let $N^{uns}_{E/K}$ be the product of the norms of primes with unstable bad reduction. By definition one has $N^0_{E/K}=N^{st}_{E/K}N^{uns}_{E/K}$. Let $K'$ be a number field extension of $K$ such that $E$ acquires semi-stable reduction everywhere over $K'$. Using equality (\ref{base change conductor}), one gets $$\hFplus(E/K)\geq \hFplus(E/K') +\frac{1}{[K':K]^2}\log N^{uns}_{E/K}.$$ As $E/K'$ has semi-stable reduction everywhere, one gets by Proposition \ref{semi stable height conductor} that $$\hFplus(E/K')\geq \frac{1}{12d'} \log N^{st}_{E/K'}.$$ Recall (use Theorem 6.2 page 413 of \cite{SiZa} or \cite{Ser} page 294) that one may choose $K'=K[E[12]]$, hence the degree $d'=[K':\mathbb{Q}]$ is controlled by the degree $d=[K:\mathbb{Q}]$ in the following way: use Lemme 4.7 of \cite{GaRe2} to get $[K':K]\leq 12^4$, hence $d'\leq 12^4 d$. Now by definition $N^{st}_{E/K'}={N^{st}_{E/K}}^{[K':K]}$ and so putting all together it gives $$\hFplus(E/K)\geq \frac{1}{12^5 d}\log N^{st}_{E/K}+ \frac{1}{12^8}\log N^{uns}_{E/K}\geq \frac{1}{12^8d}\log N^{0}_{E/K}.$$
\end{proof}

We obtain a proof of Theorem \ref{height conductor injectivity} as a consequence of Proposition \ref{height conductor} in the following way.

\begin{proof}
Apply Proposition \ref{height conductor} to get 
\begin{equation}\label{un}
\hFplus(E/K) \geq \frac{1}{12^8 d} \log N^{0}_{E/K}
\end{equation}
and apply the Matrix Lemma of Autissier's \cite{Aut} in his simplified version to the polarized elliptic curve $(E,L)$. For any $\varepsilon\in{]0,1[}$ Autissier gives (note that $2\pi^2$ disappeared with the non-negative normalization, and that $g=1$):
$$\hFplus(E/K)+\frac{1}{2}\log\frac{1}{\varepsilon}\geq\frac{(1-\varepsilon)\pi}{6d}\sum_{v\in{M_K^{\infty}}}d_v\rho(E_v,L_v)^{-2}.$$ Choose $\varepsilon =1-\frac{3}{\pi}$ to get 
\begin{equation}\label{deux}
2\hFplus(E/K)+\log\frac{\pi}{\pi-3}\geq \frac{1}{d}\sum_{v\in{M_K^{\infty}}}d_v\rho(E_v,L_v)^{-2}.
\end{equation}

Then sum the two inequalities (\ref{un}) and (\ref{deux}).
\end{proof}

\subsection{Regulator}
We start this paragraph by recalling the original conjecture of Lang-Silverman in dimension 1 (Conjecture 9.9 of \cite{Sil} page 233, see also \cite{Sil3} page 396 for a generalized version by Silverman for abelian varieties, originally the conjecture was a question of Lang about elliptic curves).

\begin{conj}({Lang-Silverman})\label{LangSilV1}
For any number field $K$, there exists a positive constant $c_{4}=c_{4}(K)$ such that for any elliptic curve $E/K$, for any point $P\in{A(K)}$, if $\mathbb{Z}\!\cdot\! P$ is Zariski dense one has
\[
\widehat{h}(P) \geq c_{4}\, \max\Big\{\hFplus(E/K),\,1\Big\}\;,
\]
where $\hat{h}(.)=\widehat{h}_{E,L}(.)$ is the N\'eron-Tate height associated to $L=3(O)$ and $\hFplus(E/K)$ is the (relative) differential height of the elliptic curve $E/K$.

\end{conj}

As $E$ is simple, the Zariski density of $\mathbb{Z}\!\cdot\! P$ is equivalent to $P$ being a non-torsion point.

\begin{rem}\label{ABC}
In the article \cite{HiSi3}, Hindry and Silverman show that the conjecture is true for a large class of elliptic curves (namely, with bounded Szpiro quotient). They furthermore show that the ABC conjecture implies the Lang-Silverman conjecture for elliptic curves.
\end{rem}

Firstly let us show (unconditionally) that the Mordell-Weil rank $m_K$ is bounded by the height $\hFplus(E/K)$. 

\begin{lem}\label{rank height}
Let $E$ be an elliptic curve defined over a number field $K$. Let $m_K$ be the rank of $E(K)$. There exists a constant $c_{5}=c_{5}(K)>0$ such that $$m_K\leq c_{5} \max\{1,\hFplus(E/K)\},$$
and one may take $c_{5}=(2^{26} 3^8+2^8\log16)d^3+2^8d\log \vert \Delta_{K/\mathbb{Q}}\vert$.
\end{lem}

\begin{proof}
We will use as a pivot the quantity $N^{0}_{E/K}$ defined as the product of the norms of the bad reduction primes of the elliptic curve $E$ over $K$. Applying Theorem 5.1 of \cite{Remond} page 775, there exists constants $c_{6}=c_{6}(K)>0$ and $c_{7}=c_{7}(K)\geq0$ such that $m_{K}\leq c_{6}(K) \log N^0_{E/K} + c_{7}(K)$. The constants are given explicitly and depend on the degree and the discriminant of the base field here. We may take the slightly bigger explicit constants $c_6(K)=d^22^{12}$ and $c_7(K)=d2^8(\log\vert \Delta_{K/\mathbb{Q}}\vert+d^2 \log16)$. This last inequality doesn't require semi-stability of $E$.
Applying Proposition \ref{height conductor} of the present text one obtains $\log N^0_{E/K}\leq c_{8}(K) \max\{\hFplus(E/K),\,1\}$, also valid in general with $c_8(K)=12^8 d$. One concludes by $$m_{K}\leq c_{6}(K) \log N^0_{E/K} + c_{7}(K)\leq c_{6}(K)c_8(K) \max\{\hFplus(E/K),\,1\}  + c_{7}(K),$$ hence $$m_K\leq (c_{6}(K)c_8(K)+c_7(K)) \max\{\hFplus(E/K),\,1\},$$ and so $c_5=12^8 2^{12}d^3+ 2^8 d(\log\vert \Delta_{K/\mathbb{Q}}\vert+d^2\log16)=(2^{26} 3^8+2^8\log16)d^3+2^8d\log \vert \Delta_{K/\mathbb{Q}}\vert$.
\end{proof}

\begin{thm}\label{reg height thm}
Assume the Lang-Silverman conjecture \ref{LangSilV1}. Let $K$ be a number field. There exists a constant $c_{10}=c_{10}(K)>0$ such that for any elliptic curve $E$ defined over $K$,
$$ \Reg(E/K)\geq \Big(c_{10} \max\{\hFplus(E/K),\,1\}\Big)^{m_K/2}.$$
Thus the set of elliptic curves defined over a fixed number field $K$ such that $E(K)$ is Zariski dense and with bounded regulator is finite (up to isomorphisms).
\end{thm}

\begin{proof}
For the duration of the proof, let us denote $h=\max\{\hFplus(E/K),\,1\}$.
The inequality is trivial for $m_K=0$ because in that case the regulator is $1$. From now on, let us assume $m_K\neq 0$. Let $L=3(O)$ and let $\hat{h}=\hat{h}_{E,L}$ be the associated canonical height on $E$ and consider the euclidean space $(E(K)\otimes\mathbb{R},\hat{h})\simeq(\mathbb{R}^{m_K}, \hat{h})$. Apply Minkowski's successive minima inequality to the Mordell-Weil lattice $\Lambda_K=E(K)/E(K)_{\mathrm{tors}}$ viewed as a lattice inside this euclidean space, 
\[
 \lambda_1(\Lambda_K)\cdots \lambda_{m_K}(\Lambda_K)\leq m_K^{m_K/2}  \Reg(E/K).
\]

\noindent Now apply $m_K$ times the inequality of Conjecture \ref{LangSilV1} to get

\begin{equation}\label{reg height part}
\Reg(E/K)\geq \frac{ c_{4}^{m_K}h^{m_K}}{m_K^{{m_K/2}}},
\end{equation}

\noindent then applying Lemma \ref{rank height} one obtains 

\begin{equation}\label{rank height sq}
\sqrt{m_K}\leq \sqrt{c_{5}} \sqrt{h}.
\end{equation}

\noindent Hence we have in (\ref{reg height part})
 
$$
 \Reg(E/K)\geq \left(c_{4}c_{5}^{-1/2} \sqrt{h}\right)^{m_K}.
$$

Finally, if the regulator is bounded then the height is bounded as soon as $m_K\neq 0$, hence the claimed finiteness.

\end{proof}

\begin{rem}
If the rank $m_K$ is zero, then the regulator $\Reg(E/K)$ is 1, so the inequality of Theorem \ref{reg height thm} is optimal in this respect. If one takes for instance the family of curves $E_p$ over $\mathbb{Q}$ defined by the affine model $y^2=x^3+p^2$, where $p$ is a prime congruent to $5$ modulo $9$, then $\rank_{\mathbb{Q}}(E_p)$ is zero by \cite{CoPa} page 399 and $\hFplus(E_p/\mathbb{Q})\gg \log p$. The situation is then similar to the CM-fields case as described in the Remark \ref{disc CM}. One cannot expect the height to be upper bounded for rank zero elliptic curves, even if the regulator is upper bounded for trivial reasons. Then when the rank is non-zero we have two cases. If the Mordell-Weil rank $m_K$ is one the inequality is exactly Lang-Silverman's. If $m_K\geq 2$ the inequality of Theorem \ref{reg height thm} is \textit{a priori} weaker than the Lang-Silverman conjecture.
\end{rem}

We thus obtain Theorem \ref{reg height thm intro} as a corollary which may be expressed in other words by: the Lang-Silverman conjecture implies that the regulator $\Reg(E/K)$ verifies a Northcott property on the set of elliptic curves defined over a fixed number field $K$ with $E(K)$ Zariski dense. This would be the elliptic curve counterpart to Theorem \ref{reg fini intro}.

\section{Conclusion}

Starting from the comparison of the residue theorem for the Dedekind zeta function and the strong form of the conjecture of Birch and Swinnerton-Dyer, one gets a way to put number fields invariants and abelian varieties invariants in link.
Studying furthermore the Brauer-Siegel theorem for number fields, Hindry gives in \cite{Hin} a conjectural equivalent in the case of abelian varieties. He then comes with a dictionary linking the invariants. Based on the previous inequalities, let us now add three lines to the (slightly reformulated) dictionary of \cite{Hin}. Note that we only focus on the case of elliptic curves here.

\begin{center}
$\begin{array}{lcccl}
& \mathrm{Number\, field\,} K &  & \mathrm{Elliptic\, curve\,} E/K &\\
\\
\mathrm{zeta\, function} & \zeta_K(s) & \leftrightarrow & L(E,s) & L\, \mathrm{function}\\
\mathrm{log \,of \,discriminant}& \log\vert D_K\vert & \leftrightarrow & \hFplus(E) & \mathrm{Faltings \,height}\\
\mathrm{regulator} & R_K & \leftrightarrow & \Reg(E/K) &\mathrm{regulator}\\
\mathrm{class\,number}& h_K &  \leftrightarrow &\vert \Sha(E/K)\vert & \mathrm{Tate}-\mathrm{Shafarevitch \,group}\\
\mathrm{torsion} & (U_K)_{\mathrm{tors}} & \leftrightarrow & (E\times \hat{E})(K)_{\mathrm{tors}}& \mathrm{torsion \,of\,} E\, \mathrm{and\, dual\,} \hat{E} \\
\\
\\

\mathrm{degree} & d & \leftrightarrow  & g & \mathrm{dimension}\\
 \mathrm{CM \,field} & r_K=r_0 & \leftrightarrow & m_K=0& E(K)\;\mathrm{ non\, Zariski \,dense}\\
  \mathrm{non-CM \,field} & r_K>r_0 & \leftrightarrow & m_K>0& E(K)\;\mathrm{ Zariski \,dense}\\
\\
\end{array}$
\end{center}

\section{A correction}

A square root is missing in the application of the Minkowski inequality on page 59 of \cite{Paz14}. We apologize and correct the statement of Theorem 4.8 page 59 of \cite{Paz14}. We keep the notation of the original article, the regulator is defined using the divisor $L=3(O)$. 

\begin{thm}\label{reg height thm}
Assume the Lang-Silverman Conjecture 4.5 page 58 of \cite{Paz14}. Let $K$ be a number field. There exists a quantity $c_{10}=c_{10}(K)>0$ only depending on $K$ such that for any elliptic curve $E$ defined over $K$ with positive rank $m_K$,
$$ \Reg(E/K)\geq \Big(\frac{c_{10} }{m_K}\max\{\hFplus(E/K),\,1\}\Big)^{m_K}.$$
Thus the set of $\overline{\mathbb{Q}}$-isomorphism classes of elliptic curves defined over a fixed number field $K$ such that $E(K)$ has positive bounded rank $m_K$ and with bounded regulator is finite.
\end{thm}

\begin{proof}
Let $L=3(O)$ and let $\hat{h}=\hat{h}_{E,L}$ be the associated canonical height on $E$ and consider the euclidean space $(E(K)\otimes\mathbb{R},\hat{h}^{1/2})\simeq(\mathbb{R}^{m_K}, \hat{h}^{1/2})$. Apply Minkowski's successive minima inequality to the Mordell-Weil lattice $\Lambda_K=E(K)/E(K)_{\mathrm{tors}}$ viewed as a lattice inside this euclidean space, 
\[
 \lambda_1(\Lambda_K)\cdots \lambda_{m_K}(\Lambda_K)\leq m_K^{m_K/2}  \Reg(E/K)^{1/2}.
\]

\noindent Now apply $m_K$ times the inequality of Conjecture 4.5 page 58 of \cite{Paz14} to get

\begin{equation}\label{reg height part}
\Reg(E/K)\geq \frac{ c_{4}^{m_K}\max\{\hFplus(E/K),\,1\}^{m_K}}{m_K^{{m_K}}},
\end{equation}

Finally, if the regulator and the rank is bounded then the height is bounded as soon as $m_K\neq 0$, hence the claimed finiteness.
\end{proof}

It implies the following new statement of Theorem 1.2 page 48 of \cite{Paz14}. 

\begin{thm}\label{reg height thm intro}
Assume the Lang-Silverman Conjecture 4.5 page 58 of \cite{Paz14}.
The set of $\overline{\mathbb{Q}}$-isomorphism classes of elliptic curves $E$, defined over a fixed number field $K$ with $E(K)$ of positive bounded rank $m_K$ and bounded regulator is finite.
\end{thm}

\section{The number field analogy and a new question}

To recover the statement without the boundedness condition, the first idea would be to improve Lemma 4.7 page 58 of \cite{Paz14} (which is still valid and unconditional) to get an inequality of the form $m_K\leq c(K)\, \hFplus(E/K)^{1-\varepsilon}$ for $\varepsilon>0$ universal and $c(K)>0$ depending only on $K$. For $K=\mathbb{Q}$ and under the Generalized Riemann Hypothesis, Mestre has obtained in II.1.2 pages 217-218 of \cite{Mes86} the inequality $m_\mathbb{Q}\ll \log N_E / \log\log N_E$ where $N_E$ is the conductor of the elliptic curve $E$. It gives hope that a stronger statement than the present Theorem \ref{reg height thm intro} could be true.

In order to remove the boundedness condition on the rank, another idea would be to prove an inequality between the regulator and the rank. In view of the number field case, it would play a role similar to Friedman's Theorem 3.4 page 53 of \cite{Paz14}. So we formulate it here as a question.

\begin{ques}
Let $E$ be an elliptic curve defined over a number field $K$. Let $m_K$ be the rank of the Mordell-Weil group $E(K)$ and let $\Reg(E/K)$ be its regulator. Can one find a positive quantity $c_0(K)$ depending only on $K$ and a strictly increasing function $f:\mathbb{N}\to\mathbb{R^+}$ such that the inequality $$\Reg(E/K)\geq c_0(K)\, f(m_K)$$ holds?
\end{ques}

Such an inequality would indeed imply that if one fixes $K$, a bounded regulator would force the rank to be bounded, exactly as in the case of number fields where a bounded regulator implies a bounded degree. This will be the subject of a future work.


\begin{thebibliography}{widest-label}

\bibitem[Aut13]{Aut} \textsc{Autissier, P.}, 
\textit{Un lemme matriciel effectif}.
Mathematische Zeitschrift {\bf 273} (2013), p. 355-361.

\bibitem[BeMa90]{BeMa} \textsc{Berg\'e, A.-M. and Martinet, J.}, 
\textit{Sur les minorations g\'eom\'etriques des r\'egulateurs}.  S\'eminaire de Th\'eorie des Nombres, Paris 1987--88, Progr. Math., Birkh\"auser Boston, Boston, MA, {\bf 81} (1990), 23--50.

\bibitem[Cha00]{Cha} \textsc{Chai, C.-L.}, 
\textit{N\'eron models for semiabelian varieties: congruence and change of base field}.
Asian J. Math. {\bf 4} (2000), 715--736.

\bibitem[CoPa09]{CoPa} \textsc{Cohen, H. and Pazuki, F.}, 
\textit{Elementary 3-descent with a 3-isogeny}. Acta Arith. {\bf 140.4} (2009), 369--404.

\bibitem[CoSi86]{CorSil} \textsc{Cornell, G. et Silverman, J. H. (editors)}, 
\textit{Arithmetic geometry}. Springer-Verlag
 {} (1986).

\bibitem[Cus84]{Cus1} \textsc{Cusick, T. W.}, 
\textit{Lower bounds for regulators.} Noordwijkerhout 1983 Proceedings, Lect. Notes Math. {\bf 1068} (1984), 63--73.

\bibitem[Cus91]{Cus2} \textsc{Cusick, T. W.}, 
\textit{The regulator spectrum for totally real cubic fields.} Monat. Math.{\bf 112.3} (1991), 217--220.

\bibitem[Fa83]{Falt} \textsc{Faltings, G.}, 
\textit{Endlichkeitss\"atze f\"ur abelsche {V}ariet\"aten \"uber {Z}ahlk\"orpern}.
Invent. Math. {\bf73} (1983), 349--366.

\bibitem[Fri89]{Fri} \textsc{Friedman, E.}, 
\textit{Analytic formulas for the regulator of a number field}. Invent. Math. {\bf 98} (1989), 599--622.

\bibitem[FrSk99]{FriSko} \textsc{Friedman, E. and Skoruppa, N.-P.},
\textit{Relative regulators of number fields}. Invent. Math.
 {\bf 135} (1999), 115--144.

\bibitem[GaR\'e14a]{GauR} \textsc{Gaudron, E. et R\'emond, G.}, \textit{Th\'eor\`eme des p\'eriodes et degr\'es minimaux d'isog\'enies}. 
Comment. Math. Helvet.  {\bf 89.2} (2014), 343--403.
 
 \bibitem[GaR\'e14b]{GaRe2} \textsc{Gaudron, E. and R\'emond, G.},
\textit{Polarisations et isog\'enies}. Duke Math.
 {\bf } (To appear, 2014).

\bibitem[Hin07]{Hin} \textsc{Hindry, M.}, 
\textit{Why is it difficult to compute the Mordell-Weil group?} Diophantine geometry, CRM Series, Ed. Norm., Pisa {\bf 4} (2007), 197--219.

\bibitem[HiSi88]{HiSi3} \textsc{Hindry, M. and Silverman, J.}, \textit{The canonical height and integral points on elliptic curves},
Invent. Math. {\bf 93} (1988), 419--450.

\bibitem[Liu02]{Liu} \textsc{Liu, Q.}, 
\textit{Algebraic Geometry and Arithmetic Curves.} Oxford Graduate Texts in Mathematics, Oxford Science Publications {\bf 6} (2002).

\bibitem[Neu99]{Neu} \textsc{Neukirch, J.}, 
\textit{Algebraic number theory.} Grundlehren der Mathematischen Wissenschaften, Springer-Verlag {\bf 322} (1999).

\bibitem[Odl77]{Odl1} \textsc{Odlyzko, A. M.}, 
\textit{Lower bounds for discriminants of number fields. II.} T\^ohoku Math. J. {\bf 29.2} (1977), 209--216.

\bibitem[Odl90]{Odl2} \textsc{Odlyzko, A. M.}, 
\textit{Bounds for discriminants and related estimates for class numbers, regulators and zeros of zeta functions: a survey of recent results}. S\'em. Th\'eor. Nombres Bordeaux {\bf 2.2.1} (1990), 119--141.

\bibitem[Paz10]{Paz2} \textsc{Pazuki, F.}, 
\textit{Remarques sur une conjecture de {L}ang}. Journal de Théorie des Nombres de Bordeaux {\bf 22} no.1 (2010), 161--179.

\bibitem[Paz12]{Paz3} \textsc{Pazuki, F.}, 
\textit{Theta height and Faltings height}. Bull. Soc. Math. France {\bf 140.1} (2012), 19--49.

\bibitem[Rem52]{Rema} \textsc{Remak, R.}, 
\textit{\"Uber Gr\"ossenbeziehungen zwischen Diskriminante und Regulator eines algebraischen Zahlk\"orpers}. Compositio Math. {\bf 10} (1952), 245--285.

\bibitem[R\'em05]{Remond2} \textsc{R\'emond, G.}, 
\textit{In\'egalit\'e de Vojta g\'en\'eralis\'ee.} Bull. Soc. Math. France {\bf 133.4} (2005), 459--495.

\bibitem[R\'em10]{Remond} \textsc{R\'emond, G.}, 
\textit{Nombre de points rationnels des courbes.} Proc. Lond. Math. Soc. {\bf 101.3} (2010), 759--794.

\bibitem[Sam03]{Sam} \textsc{Samuel, P.}, 
\textit{Th\'eorie alg\'ebrique des nombres}. Hermann, Paris, \'edition revue et corrig\'ee {\bf } (2003).

\bibitem[SGA72]{SGA} \textsc{Grothendieck, A.}, 
\textit{Groupes de monodromie en g\'eom\'etrie alg\'ebrique}. SGA 7.1, Lecture Notes in Mathematics, Springer-Verlag, {\bf 288} (1972).

\bibitem[Se72]{Ser} \textsc{Serre, J.-P.},
\textit{Propri\'et\'es galoisiennes des points d'ordre fini des courbes elliptiques}.
Invent. Math. {\bf15} (1972), 259--331.

\bibitem[SiZa95]{SiZa} \textsc{Silverberg, A. and Zarhin, Yu.},
\textit{Semistable reduction and torsion subgroups of abelian varieties}.
Ann. Inst. Fourier {\bf45} (1995), 403--420.

\bibitem[Si86]{Sil} \textsc{Silverman, J. H.},
\textit{Arithmetic of elliptic curves}.
Springer GTM {\bf106} (1986), second printing of the first edition.

\bibitem[Si84b]{Sil3} \textsc{Silverman, J. H.},
\textit{Lower bounds for height functions}.
Duke Math. J. {\bf51} (1984), 395--403.

\bibitem[Si84a]{Sil4} \textsc{Silverman, J. H.},
\textit{An inequality relating the regulator and the discriminant of a number field}.
Journal of Number Theory {\bf19.3} (1984), 437--442.

\bibitem[Wash97]{Wash} \textsc{Washington, L.}, 
\textit{Introduction to cyclotomic fields}. Springer, GTM {\bf 83} (1997), second edition.

\bibitem[Zim81]{Zim} \textsc{Zimmert, R.}, 
\textit{Ideale kleiner Norm in Idealklassen und eine Regulatorabsch\"atzung}. Invent. Math. {\bf 62.3} (1981), 367--380.



\end{thebibliography}

\begin{thebibliography}{widest-label}

\bibitem[Mes86]{Mes86} \textsc{Mestre, J.-F.}, 
\textit{Formules explicites et minorations de conducteurs de vari\'et\'es alg\'ebriques}. Compositio Math. {\bf 58.2} (1986), 209--232.

\bibitem[Paz14]{Paz14} \textsc{Pazuki, F.}, 
\textit{Heights and regulators of number fields and elliptic curves}. Publ. Math. Besan\c{c}on {\bf 2014/2} (2014), 47--62.

\end{thebibliography}
\end{document}